\documentclass[a4paper,12pt,reqno]{amsart}
\usepackage{latexsym}
\usepackage{amssymb} 
\usepackage{mathrsfs}
\usepackage{amsmath}
\usepackage{mathtools}
\usepackage{latexsym}
\usepackage{delarray}
\usepackage{xcolor}
\usepackage{enumitem}
\usepackage{amssymb,amsmath,amsfonts,amsthm,mathrsfs}
\renewcommand\eqref[1]{(\ref{#1})} 
\renewcommand\eqref[1]{(\ref{#1})}

\makeatletter
\newcommand*{\mint}[1]{%
  \mint@l{#1}{}%
}
\newcommand*{\mint@l}[2]{%
  \@ifnextchar\limits{%
    \mint@l{#1}%
  }{%
    \@ifnextchar\nolimits{%
      \mint@l{#1}%
    }{%
      \@ifnextchar\displaylimits{%
        \mint@l{#1}%
      }{%
        \mint@s{#2}{#1}%
      }%
    }%
  }%
}
\newcommand*{\mint@s}[2]{%
  \@ifnextchar_{%
    \mint@sub{#1}{#2}%
  }{%
    \@ifnextchar^{%
      \mint@sup{#1}{#2}%
    }{%
      \mint@{#1}{#2}{}{}%
    }%
  }%
}
\def\mint@sub#1#2_#3{%
  \@ifnextchar^{%
    \mint@sub@sup{#1}{#2}{#3}%
  }{%
    \mint@{#1}{#2}{#3}{}%
  }%
}
\def\mint@sup#1#2^#3{%
  \@ifnextchar_{%
    \mint@sup@sub{#1}{#2}{#3}%
  }{%
    \mint@{#1}{#2}{}{#3}%
  }%
}
\def\mint@sub@sup#1#2#3^#4{%
  \mint@{#1}{#2}{#3}{#4}%
}
\def\mint@sup@sub#1#2#3_#4{%
  \mint@{#1}{#2}{#4}{#3}%
}
\newcommand*{\mint@}[4]{%
  \mathop{}%
  \mkern-\thinmuskip
  \mathchoice{%
    \mint@@{#1}{#2}{#3}{#4}%
        \displaystyle\textstyle\scriptstyle
  }{%
    \mint@@{#1}{#2}{#3}{#4}%
        \textstyle\scriptstyle\scriptstyle
  }{%
    \mint@@{#1}{#2}{#3}{#4}%
        \scriptstyle\scriptscriptstyle\scriptscriptstyle
  }{%
    \mint@@{#1}{#2}{#3}{#4}%
        \scriptscriptstyle\scriptscriptstyle\scriptscriptstyle
  }%
  \mkern-\thinmuskip
  \int#1%
  \ifx\\#3\\\else_{#3}\fi
  \ifx\\#4\\\else^{#4}\fi
}
\newcommand*{\mint@@}[7]{%
  \begingroup
    \sbox0{$#5\int\m@th$}%
    \sbox2{$#5\int_{}\m@th$}%
    \dimen2=\wd0 %
    \let\mint@limits=#1\relax
    \ifx\mint@limits\relax
      \sbox4{$#5\int_{\kern1sp}^{\kern1sp}\m@th$}%
      \ifdim\wd4>\wd2 %
        \let\mint@limits=\nolimits
      \else
        \let\mint@limits=\limits
      \fi
    \fi
    \ifx\mint@limits\displaylimits
      \ifx#5\displaystyle
        \let\mint@limits=\limits
      \fi
    \fi
    \ifx\mint@limits\limits
      \sbox0{$#7#3\m@th$}%
      \sbox2{$#7#4\m@th$}%
      \ifdim\wd0>\dimen2 %
        \dimen2=\wd0 %
      \fi
      \ifdim\wd2>\dimen2 %
        \dimen2=\wd2 %
      \fi
    \fi
    \rlap{%
      $#5%
        \vcenter{%
          \hbox to\dimen2{%
            \hss
            $#6{#2}\m@th$%
            \hss
          }%
        }%
      $%
    }%
  \endgroup
}
\numberwithin{equation}{section}
\theoremstyle{plain}
 \newtheorem{thm}{Theorem}[section]
 \newtheorem{cor}[thm]{Corollary}
 \newtheorem{lem}[thm]{Lemma}
 \newtheorem{prop}[thm]{Proposition}
 \theoremstyle{definition}
 \newtheorem{defn}[thm]{Definition}
 \theoremstyle{remark}
 \newtheorem{rem}[thm]{Remark}
 \newtheorem{ex}{Example}
 
 \numberwithin{equation}{section}
\newcommand{\half}{\frac{1}{2}}
\newcommand{\qf}{\frac{q}{4}}
\newcommand{\qt}{\frac{q}{2}}

\newcommand{\ene}{\mathbb{N}}
\newcommand{\ar}{\mathbb{R}}

\newcommand{\arn}{{\mathbb{R}}^n}
\newcommand{\Rn}{{\mathbb{R}}^n}

\newcommand{\ardn}{{\mathbb{R}}^{n}\times{\mathbb{R}}^{n}}

\newcommand{\bi}{\begin{itemize}}
\newcommand{\ei}{\end{itemize}}
\newcommand{\be}{\begin{enumerate}}
\newcommand{\ee}{\end{enumerate}}
\newcommand{\beq}{\begin{equation}}
\newcommand{\eq}{\end{equation}}

\newcommand{\jp}{\langle\xi\rangle}
\newcommand{\jpx}{\langle x\rangle}

\DeclareMathOperator{\os}{o}

\title[ON A CLASS OF ANHARMONIC OSCILLATORS II. GENERAL CASE]{ON A CLASS OF ANHARMONIC OSCILLATORS II. GENERAL CASE}
\author[M. Chatzakou]{Marianna Chatzakou}
\address{
	Marianna Chatzakou:
	\endgraf
    Department of Mathematics: Analysis, Logic and Discrete Mathematics
    \endgraf
    Ghent University, Belgium
  	\endgraf
	{\it E-mail address} {\rm marianna.chatzakou@ugent.be}
		}
	
\author[J. Delgado]{Julio Delgado}
\address{
  Julio Delgado:
  \endgraf
  Departamento de Matem\'aticas
  \endgraf
  Universidad del Valle, Colombia
  \endgraf
	{\it E-mail address} {\rm delgado.julio@correounivalle.edu.co} }

\author[M. Ruzhansky]{Michael Ruzhansky}
\address{
  Michael Ruzhansky:
  \endgraf
  Department of Mathematics: Analysis, Logic and Discrete Mathematics
  \endgraf
  Ghent University, Belgium
  \endgraf
 and
  \endgraf
  School of Mathematical Sciences
  \endgraf
  Queen Mary University of London
  \endgraf
  United Kingdom
  \endgraf
  {\it E-mail address} {\rm michael.ruzhansky@ugent.be}
  }

\begin{document}

\thanks{The authors are supported by the FWO Odysseus 1 grant G.0H94.18N: Analysis and Partial Differential Equations and by the Methusalem programme of the Ghent University Special Research Fund (BOF) (Grant number 01M01021). Julio Delgado is also supported by Vic. Inv Universidad del Valle.  Grant No.  CI-71281.  Michael Ruzhansky is also supported by EPSRC grants EP/R003025/2 and EP/V005529. }

\begin{abstract} 
 In this work we study a class of anharmonic oscillators on $\arn$ corresponding to  Hamiltonians of the form $A(D)+V(x)$,  where $A(\xi)$   and $V(x)$ are   
 $C^{\infty}$ functions enjoying some regularity conditions. Our class includes fractional relativistic Schr\"odinger operators and anharmonic oscillators with fractional potentials. By associating  a H\"ormander metric we obtain spectral properties in terms of Schatten-von Neumann classes for their negative powers and derive from them estimates on the rate of growth for the eigenvalues of the operators $A(D)+V(x)$.  This extends the analysis in the first part \cite{CDR}, where the case of polynomial $A$ and $V$ has been analysed.
\end{abstract}

\maketitle

\tableofcontents

\section{Introduction}
In this manuscript we obtain  spectral properties and in particular estimates for the rate of growth of eigenvalues for a class of operators on $\Rn$ of the form  $A(D)+V(x)$, where $A(\xi)$ and $V(x)$ are appropriate smooth functions. Some important examples in this class  include fractional relativistic Schr\"odinger operators, the special cases of  relativistic Schr\"odinger operators $\sqrt{I-\Delta}+V(x)$, and anharmonic oscillators with fractional power potentials $(-\Delta)^{\ell}+\jpx^{2\kappa}$, where $\ell$ is a positive integer and  $\kappa>0$. Our class also allows lower order terms with respect to each symbol $A(\xi)$ and $V(x)$; notably real-valued potentials with a finite range of negative values. Hamiltonians in relativistic quantum mechanics, and quantum field theory are not, in general, partial
differential operators (as in nonrelativistic quantum mechanics), but pseudo-differential operators. Herein, we frame our class within the setting of the Weyl-H\"ormander calculus by introducing a suitable H\"ormander metric and obtaining a substantial extension of the class considered in \cite{CDR}.\\
 
The mother case of the Hamiltonian  $-\Delta+V(x)$ $(\ell=1)$, corresponding to the Schr\"odinger equation is one of the milestone  objects of study in mathematical physics, specially since the the study of energy levels for the Schr\"odinger equation is reduced to the eigenvalue problem associated the Hamiltonian. 
 One of the most basic examples of anharmonic oscillators of  polynomial  form $-\Delta+|x|^{2k}$, is the one dimensional quartic oscillator ($n=1, \, k=2$) which is an important model in quantum physics and has been intensively studied in the last 50 years. However, the exact solution for the eigenvalue problem of this model is unknown (cf. \cite{OU}, \cite{BB}). It is clear that the situation is even more subtle for non-polynomial Hamiltonians. This kind of fact is a further motivation for the investigation of different  approximative and qualitative properties for the operators $A(D)+V(x)$. \\

The case of polynomial Hamiltonians has been widely studied. A class  of anharmonic oscillators  arises in the form $$-\frac{d^{2\ell}}{dx^{2\ell}}+x^{2k}+p(x),$$ where $p(x)$ is a polynomial of order $2k-1$ on $\ar$ and with $k, \ell$ integers $\geq 1$. The spectral asymptotics of such operators  have been analysed by   B. Helffer and D. Robert \cite{HD81, HD82, H}. The authors have recently studied  anharmonic oscillators  on  $\arn$ in \cite{CDR}, where a prototype operator is of the form
 \beq (-\Delta)^{\ell}+|x|^{2k} \label{gahw1},\eq
where $k,\ell$ are integers $\geq 1$. By considering $A(\xi)=|\xi|^{2\ell}$ and $V(x)=|x|^{2k}$, these symbols will be absorbed by our new class considered in this paper.   Spectral properties for the  type \eqref{gahw1} in the case $k=\ell$ can be found in \cite{H} and  \cite{Rob}.  The case of a fractional Laplacian and a quartic potential has been considered by S. Durugo and  J. L\"orinczi in \cite{DL}. Bochner–Riesz means and spectral analysis for the one-dimensional anharmonic oscillator 
$-\frac{d^{2\ell}}{dx^{2\ell}}+|x|$ has been recently studied in \cite{CHS}. Physical models related to the operator $\sqrt{-\Delta+m^2} $ have been intensely studied in the last 30 years and there exists a huge literature on the spectral properties of relativistic Hamiltonians, most of it strongly influenced by the works of Lieb  on the stability of relativistic matter (cf. \cite{LY88}, \cite{LY87}, \cite{LL}). On the other hand, the study of general  fractional relativistic Schr\"odinger operators $(I-\Delta)^{\gamma}+V(x)$,  $\gamma >0,$ has also attracted  the interest in the last decades. The operators   $(I-\Delta)^{\gamma}$ are also 
related to the so-called {\it relativistic} $\gamma$-stable process. See for  instance  \cite{Amb}, \cite{FF} and the references therein for recent investigation on the fractional relativistic Schr\"odinger operators.

Herein we study the operator $A(D)+V(x)$ within the setting of H\"ormander's $S(m,g)$ classes by introducing a H\"ormander metric $g^{(A,V)}$ in Section \ref{SEC:anharmonic}. In Section \ref{SEC:anharmonic-notes} we give an intrinsic formulation of such classes:

\smallskip
{\em  For $m\in\ar$,  appropriate functions $A(\xi), V(x)$ on $\Rn$  and suitable values of $\gamma$ and $\kappa$, the class $\Sigma_{A,V}^m$
 consists of all $C^{\infty}(\ardn)$ functions enjoying the property  that for all multi-indexes $\alpha, \beta$ there exists a constant $C_{\alpha\beta}$ such that 
 \beq\label{sigmaclxx}|\partial_{x}^{\beta}\partial_{\xi}^{\alpha}a(x,\xi)|\leq C_{\alpha\beta}(q+V(x)+A(\xi))^{m-\frac{|\beta|}{2\kappa}-\frac{|\alpha|}{2\gamma}}\,,\eq
 holds all $x,\xi\in\mathbb R^n$.}
 
 \medskip
By looking at the negative powers of our operators $A(D)+V(x)$ and studying the corresponding  Schatten-von Neumann properties within the setting of H\"ormander $S(m,g)$ classes, we will deduce the rate of decay of the eigenvalues for those negative powers. Thus, from the inverses of these negative powers we obtain estimates for the rate of growth of eigenvalues of the operators $A(D)+V(x)$.  The investigation 
of such properties within these  classes started with H\"ormander \cite{Hor79}.  Other works on Schatten-von Neumann classes within the Weyl-H\"ormander calculus can be found in \cite{Tof06}, \cite{Tof08}. See also \cite{DM14}, \cite{DM14b}, \cite{DM17} for several symbolic and kernel criteria on different types of domains.  For the spectral theory of non-commutative versions of the harmonic oscillator we refer the reader to the  works of Parmeggiani et al \cite{P06}, \cite{P08}, \cite{P10}, \cite{P14}, \cite{PW01} and \cite{PW02}.\\

\medskip
The main results of this work give the order of the corresponding Schatten-von Neumann class for the negative powers of the operators $A(D)+V(x)$.  The special case of the trace class is also distinguished. That is the contents of Theorem \ref{sch1mf} and Corollary \ref{cor.b.1}. From those we derive estimates for the rate of growth of  our operators $A(D)+V(x)$ in \eqref{EQ:growthb}.  In Theorem  \ref{sch1mfff} we give a sharp condition to guarantee the membership of the negative powers of our main examples  to the  Schatten-von Neumann classes. At the end of the section we also provide examples for the special case of the fractional relativistic Schr\"odinger operators with fractional potentials.

\section{Weyl-H{\"o}rmander calculus and Schatten-von Neumann classes}

In this section we  briefly review some basic elements of the Weyl-H{\"o}rmander calculus and the Schatten-von Neumann classes.  For a comprehensive study on the Weyl-H{\"o}rmander calculus,  we refer the interested reader to  \cite{Hor85}, \cite{Ler}, \cite{BL}. \\
 
 Let us briefly recall the main concepts and ideas that will be involved in this work. Before doing so, let us note that, in the sequel, we shall use capital latin letters, e.g. $X,Y,T$ to denote elements in the phase space. These do, in turn, correspond to pairs of the form $(x,\xi),(y,\eta),(t,\tau)$, where Latin and Greek letters have been used for elements in  the configuration space $\mathbb{R}^n$, and its dual $(\mathbb{R}^{n})^{*}$, respectively. \\
 
\textit{ For $a=a(x,\xi) \in S'(\ardn)$ ($x\in \mathbb{R}^{n}$ and
$\xi\in \mathbb{R}^{n}$) and $t\in\ar$, we define the t-quantization of the symbol $a$ as being the operator $a_t(x,D):S(\mathbb{R}^{n})\rightarrow S(\mathbb{R}^{n})$ given by the formula \[a_t(x,D)u(x)=(2\pi)^{-n}\int_{\arn}\int_{\arn} e^{i(x-y)\xi}a(tx+(1-t)y,\xi)u(y)dyd\xi.\]
}\\

The operator $a_1(x,D)$, or for simplicity $a(x,D)$, is known as the {\em Kohn-Nirenberg  quantization}, while the operator $a_{\frac{1}{2}}(x,D)$, denoted also by $a^w(x,D)$, is called the {\em Weyl quantization}. In the particular cases of operators that we consider in this work, the corresponding symbols that appear stay invariant (up to lower under terms) under different quantizations.\\

To define Weyl-H\"ormander classes of symbols the following notions shall be recalled; if $g_X(\cdot)$ is a positive definite quadratic form on the phase space, then \textit{$g_\cdot(\cdot)$ is a H\"ormander's metric} if the following three conditions are satisfied:
\begin{enumerate}
\item {\bf Continuity or slowness}- There exist a constant $C>0$ such that 
\[g_{X}(X-Y)\leq C^{-1}\implies \left(\frac{g_X(T)}{g_Y(T)}\right)^{\pm 1}\leq 1\,,\quad \forall T\,.\]
\item {\bf Uncertainty principle}- We say that $g$ satisfies the {\em uncertainty principle }
if \[\lambda_{g}(X)=\inf_{T\neq 0}
\left(\frac{g_{X}^{\sigma}(T)}{g_{X}(T)}\right)^{1/2}\geq 1\,,\quad \forall X,T\,,\]
where $g_{X}^{\sigma}(T):=\sup_{W\neq 0}
\left\{\sigma(Y,T)^{2}/g_{X}(Y)\right\}$, and $\sigma(Y,T):=t \cdot \eta-y\cdot \tau$.
\item {\bf Temperateness}- We say  that $g$ is temperate if there exist $\overline{C}>0$ and $J\in \mathbb{N}$ such that
\begin{equation}\label{tempc1x}
  \left(\frac{g_{X}(T)}{g_{Y}(T)}\right)^{\pm1}\leq \overline{C}(1+g_{Y}^{\sigma}(X-Y))^{J}\,,\quad \forall X,Y,T\,.
\end{equation}
\end{enumerate}
For $g$ being a fixed H\"ormander metric, and $M$ being a positive function on the phase space we say that \textit{$M$ is a $g$-weight}, if the following are satisfied:
\begin{enumerate}
    \item {\bf $g$- continuous}, if there exists $\tilde{C}>0$
such that
\[g_{X}(X-Y)\leq \frac{1}{\tilde{C}}\implies\left( \frac{M(X)}{M(Y)}\right)^{\pm1}\leq \tilde{C}\,,\quad \forall X,Y\,.\]
\item {\bf $g$-temperate}, if there exist $\tilde{C}>0$ and $N\in \mathbb{N}$ such that
\[\left( \frac{M(X)}{M(Y)}\right)^{\pm1}\leq \tilde{C}(1+g_{Y}^{\sigma}(X-Y))^{N}\,,\quad \forall X,Y\,.\]

\end{enumerate}
For $M,g$ being as above, we define the \textit{set of symbols $S(M,g)$} as the set of smooth functions $a$ on the phase space such that for any integer $k$ there exists $C_{k}>0$, such that for all
$X,T_{1},...,T_{k}\in \ardn$ we have 
\beq |a^{(k)}(X;T_{1},...,T_{k})|\leq C_{k}M(X)\prod_{i=1}^{k} g_{X}^{1/2}(T_{i}) .\label{inwhk}\eq
The notation $a^{(k)}$ stands for the $k^{th}$ derivative of $a$ and  $a^{(k)}(X;T_{1},...,T_{k})$ denotes the $k^{th}$  derivative of $a$ at $X$ in the directions $T_{1},...,T_{k}$. 
For $a\in S(M,g)$ we define
\[
\|a\|_{k,S(M,g)}:=\inf \{C_k:C_k \quad \text{satisfies}\quad \eqref{inwhk} \}\,.
\]
The family of seminorms $\parallel \cdot\parallel_{k,S(M,g)}$ endows $S(M,g)$ with the topology of a Fr\'echet space.\\

The family of seminorms $\parallel \cdot\parallel_{k,S(M,g)}$ endows $S(M,g)$ with the topology of a Fr\'echet space.\\

Finally let us discuss some basic definitions of the \textit{Schatten-von Neumann classes of operators}; for a more detailed exposition of the theory we refer to \cite{GK}, \cite{RS}, \cite{Sim}, \cite{Sch}. \\

Let $H$ be a separable Hilbert space over $\mathbb{C}$ endowed with the inner product $(\cdot,\cdot)$, and let $T$ be some compact (linear) operator from $H$ to itself. We denote by $|T|:=(T^{*}T)^{1/2}$ the \textit{absolute value of $T$}, and by $s_n(T)$ the \textit{singular values of $T$}; that is the eigenvalues of $|T|$, that correspond to the eigenfunctions of the latter given by the spectral theorem. The operator $T$ belongs to the \textit{Schatten-von Neumann class of operators $S_{p}(H)$}, where $1\leq p< \infty$, if
\[
\|T\|_{S_p}:=\left( \sum_{k=1}^{\infty}(s_{k}(T))^{p} \right)^{\frac{1}{p}}<\infty\,.
\]
The space $S_{p}$ is a Banach space if endowed with the natural norm $\|\cdot\|_{S_p}$. In particular, the Banach space $S_{1}(H)$ is the space of \textit{trace-class operators}, while for $T\in S_1$ the quantity
\[
\textnormal{Tr}(T):=\sum_{n=1}^{\infty}(T\phi_n,\phi_n)\,,
\]
where $(\phi_n)$ is an orthonormal basis in $H$, is well-defined and shall be called the \textit{trace $\textnormal{Tr}(T)$ of $T$}. Moreover, the space $S_2(H)$ is identified with the space of \textit{Hilbert-Schmidt operators} on $H$.

\section{The class of operators and the metric}

\label{SEC:anharmonic}

In this section we start with the study of the specific class of operators that we consider here that are regarded to be of the form $A(D)+V(x)$. We first introduce a suitable class of functions for the terms  $A(\xi), V(x)$ of their symbols.

\begin{defn}\label{class1a}
Let $\Gamma:\Rn\rightarrow \ar $ be a continuous function. We say that $\Gamma$ is a \textit{$\tau$-function}, for some $\tau>0$, if the following conditions are satisfied:
\begin{enumerate}[label=(\alph*)]
\item \label{itm:gfun.a} there exists $q>0$ such that $\Gamma(x)+\qf>0$ for all $x\in\Rn$;
\item\label{itm:gfun.b} there exist $\tau>0$ such that 
\[C_1|(\qt+\Gamma(x))^{\frac{1}{2\tau}}-(\qt+\Gamma(y))^{\frac{1}{2\tau}}|\leq |x-y|, \, \mbox{ for all }|x|, |y|\geq R\,,\]
for some $C_1>0, R>0$.
\end{enumerate}
\end{defn}

\begin{ex}\label{exsy3} As some examples of functions satisfying Definition \ref{class1a} we can consider potentials of the form:
\begin{itemize}
\item[(i)]  $V(x)=\jpx^{2\kappa}+\, $ lower order, with $\kappa>0$ and choosing $\tau=\kappa$;
\item[(ii)] $V(x)=|x|^{2k}+\, $ lower order, where $k$ is an integer $\geq 1$, and choosing $\tau=k$.

Similarly for the symbol $A(\xi)$ of the operator $A(D)$ we can consider:
\item[(i)]  $A(\xi)=\jp^{2\gamma}+\, $ lower order,  with $\tau=\gamma>0$;
\item[(ii)] $A(\xi)=|\xi|^{2\ell}+\, $ lower order, where $\ell$ is an integer $\geq 1$ and $\tau=\ell$.

\end{itemize}
\end{ex}
More generally we define the following class of  operators we are going to consider. 

\begin{defn}{The $(\gamma, \kappa)$-class}\label{gtclass} Let $\gamma,\kappa>0$. We say that an operator of the form $A(D)+V(x)$ belongs to the  $(\gamma, \kappa)$-class, if $A(\xi)$ is a $\gamma$-function, and $V(x)$ is a $\kappa$-function.
\end{defn}

\begin{ex} Special examples of operators in the $(\gamma, \kappa)$-class. 

\begin{itemize}
\item[(i)]  \textit{Anharmonic oscillators} of the form
\[(-\Delta)^{\ell} + |x|^{2k},\]
where $k, \ell$ integers $\geq 1$, belong to the $(\ell,\kappa)$-class.\\

Spectral properties for this class of operators have been recently analysed by the authors in \cite{CDR} within the setting of Weyl-H\"ormander calculus. Another example of operators in the $(\ell,\kappa)$-class
are those expressed by the formula $-\frac{d^{2\ell}}{dx^{2\ell}}+x^{2k}+p_1(x)$, where $p_1$ is a suitable polynomial of order $2k-1$. Those have been studied by Helffer and Robert (cf. \cite{HD82}, \cite{HD82b}).
\item[(ii)] \textit{Relativistic Schr\"odinger operators} of the form
\[\sqrt{I-\Delta}+V(x),\]
where $V(x)$ satisfies the Definition \ref{class1a} for a suitable $\kappa>0$, belong the the $(1,\kappa)$-class.

\item[(iii)] \textit{Anharmonic oscillators with fractional potential}
\[(-\Delta)^{\ell} + \jpx^{2\kappa},\]
where $\ell$ is an integer $\geq 1$ and $\kappa>0$, are in the $(\ell,\kappa)$-class.
\end{itemize}
\end{ex}

We associate to an operator $A(D)+V(x)$ in the $(\gamma, \kappa)$-class, the following metric: 
\beq g^{(A,V)}=\frac{dx^2}{(q+V(x)+A(\xi))^{\frac{1}{\kappa}}}+\frac{d\xi ^2}{(q+V(x)+A(\xi))^{\frac{1}{\gamma}}}\label{anhmet012x}\eq
We choose a constant $q>0$ in \eqref{anhmet012x} to be  such that $q+V(x)+A(\xi)\geq 1$, for all $x,\xi \in \mathbb{R}^n$.

Our first aim is to prove that the metric introduced in \eqref{anhmet012x} is a H\"ormander metric. To this end, let us first recall the auxiliary lemma that will be used later in the proof of the continuity property. For the proof of it, the interested reader, can consult  \cite[Theorem 18.4.2]{Hor85}.
\begin{lem}\label{resl} Let $g$ be a Riemannian metric on the phase space. The following statements are equivalent:
\begin{enumerate}[label=(\roman*)]
    \item $g$ is continuous.
    \item \label{itm: cont.l.ii} There exists a constant $C\geq 1$ such that
\[g_X(X-Y)\leq C^{-1}\,\,\mbox{ implies }\,\, g_Y\leq Cg_X.\]
\item here exists a constant $C\geq 1$ such that
\[g_X(Y)\leq C^{-1}\,\,\mbox{ implies }\,\, g_{X+Y}\leq Cg_X.\]
\end{enumerate}
\end{lem} 

We now consider the metric  $g^{(A,V)}$ in detail.
\begin{thm}\label{contgk} The metric $g=g^{(A,V)}$ defined by \eqref{anhmet012x} is a H\"ormander metric.
\end{thm} 
\begin{proof} Regarding the uncertainty parameter we have
	\[\lambda_{g}(X)=(q+V(x)+A(\xi))^{\frac{\kappa+\gamma}{2\kappa\gamma}}\, ,\]
 and it is clear that, due the choice of $q$ in \eqref{anhmet012x}, we have  $\lambda_{g}(X)\geq 1$ for all $X$.\\

For the proof of the continuity of the metric $g^{(A,V)}$, we will use the characterisation \ref{itm: cont.l.ii} in Lemma \ref{resl}; i.e., we aim to prove that 
\[\frac{|x-y|^2}{(q+V(x)+A(\xi))^{\frac{1}{\kappa}}}+\frac{|\xi-\eta|^2}{(q+V(x)+A(\xi))^{\frac{1}{\gamma}}}\leq C^{-1}\implies\]
\begin{multline*}
\frac{|t|^2}{(q+V(y)+A(\eta))^{\frac{1}{\kappa}}}+\frac{|\tau|^2}{(q+V(y)+A(\eta))^{\frac{1}{\gamma}}} \\
\leq C\left(\frac{|t|^2}{(q+V(x)+A(\xi))^{\frac{1}{\kappa}}}+\frac{|\tau|^2}{(q+V(x)+A(\xi))^{\frac{1}{\gamma}}}\right),
\end{multline*}
for all $X,Y,T \in\ardn$.\\

The proof can be reduced to proving that 
\[\frac{|x-y|^2}{(q+V(x)+A(\xi))^{\frac{1}{\kappa}}}\, ,\,\, \frac{|\xi-\eta|^2}{(q+V(x)+A(\xi))^{\frac{1}{\gamma}}}\leq C^{-1}\implies 
\frac{(q+V(x)+A(\xi))^{\frac{1}{\kappa}}}{(q+V(y)+A(\eta))^{\frac{1}{\kappa}}}\leq C .\]
Or, even simpler, to proving the following implication
\beq\frac{|x-y|}{(q+V(x)+A(\xi))^{\frac{1}{2\kappa}}}\, ,\,\, \frac{|\xi-\eta|^2}{(q+V(x)+A(\xi))^{\frac{1}{2\gamma}}}\leq C^{-1}\implies 
\frac{q+V(x)+A(\xi)}{q+V(y)+A(\eta)}\leq C .\label{lhd1t}\eq

We will assume that the LHS of \eqref{lhd1t} holds for a constant $C>0$ to be chosen later on. Since \[C_1|(\qt+V(x))^{\frac{1}{2\kappa}}-(\qt+V(y))^{\frac{1}{2\kappa}}|\leq |x-y|, \, \mbox{ for all }|x|, |y|\geq R;\]
and analogously for $A(\xi)$, we have  
\beq C_1(\qt+V(x))^{\frac{1}{2\kappa}}\leq  C^{-1}(q+V(x)+A(\xi))^{\frac{1}{2\kappa}}+C_1(\qt+V(y))^{\frac{1}{2\kappa}},\label{le23}
\eq
for  all $|x|, |y|\geq R$, and 
\beq C_1(\qt+A(\xi))^{\frac{1}{2\gamma}}\leq  C^{-1}(q+V(x)+A(\xi))^{\frac{1}{2\gamma}}+C_1(q+A(\eta))^{\frac{1}{2\gamma}}.\label{le24}
\eq
for all $|\xi|, |\eta|\geq R$. We note that due to the continuity of $A$ and $V$ we can also assume that the above inequalities hold for $|x|, |y|\leq 2R$ and $|\xi|, |\eta|\leq 2R$.\\

By taking powers $2\kappa$ and $2\gamma$ of \eqref{le23} and \eqref{le24} respectively, there exists a constant $C_2>1$  only dependent on $\kappa$ and $\gamma$ such that
\beq V(x)\leq  C^{-2\kappa}C_2(q+V(x)+A(\xi))+C_2V(y),\label{le231b}
\eq

\beq A(\xi)\leq   C^{-2\gamma}C_2(q+V(x)+A(\xi))+C_2A(\eta).\label{le241x}
\eq
By adding \eqref{le231b} and \eqref{le241x},  taking $\kappa_0=\min\{\kappa,\gamma\}$ and $C$ large enough, we obtain \[q+V(x)+A(\xi)\leq  \tilde{C}( q+V(y)+A(\eta))\,,\]
for some $\tilde{C}>0$. Hence, we obtain
\[\frac{q+V(x)+A(\xi)}{q+V(y)+A(\eta)}\leq \tilde{C}\,,\]
and the last inequality shows the continuity of the metric $g^{(A,V)}$.\\ 

We now prove the temperateness. According to \eqref{tempc1x}, we need to prove that there exist $C>0$ and $N\in\ene$ such that

\begin{multline*}
\frac{|t|^2}{(q+V(x)+A(\xi))^{\frac{1}{\kappa}}}+\frac{|\tau|^2}{(q+V(x)+A(\xi))^{\frac{1}{\gamma}}} \\
\leq C\left(\frac{|t|^2}{(q+V(y)+A(\eta))^{\frac{1}{\kappa}}}+\frac{|\tau|^2}{(q+V(y)+A(\eta))^{\frac{1}{\gamma}}}\right)\times \\
 \times\left(1+(q+V(x)+A(\xi))^{\frac{1}{\gamma}}|x-y|^2+(q+V(x)+A(\xi))^{\frac{1}{\kappa}}|\xi-\eta|^2\right)^{N},
\end{multline*}
and 
 \begin{multline*}
 \frac{|t|^2}{(q+V(y)+A(\eta))^{\frac{1}{\kappa}}}+\frac{|\tau|^2}{(q+V(y)+A(\eta))^{\frac{1}{\gamma}}} \\ \leq
 C\left(\frac{|t|^2}{(q+V(x)+A(\xi))^{\frac{1}{\kappa}}}+\frac{|\tau|^2}{(2+V(x)+A(\xi))^{\frac{1}{\gamma}}}\right)\times \\
\times\left(1+(q+V(x)+A(\xi))^{\frac{1}{\gamma}}|x-y|^2+(q+V(x)+A(\xi))^{\frac{1}{\kappa}}|\xi-\eta|^2\right)^{N},
  \end{multline*}
 for all $t, \tau\in\arn$.\\
 
We will only prove the first inequality since the second one can be proven in a similar way.\\  
 
We now observe that we can reduce the proof of the first inequality to the proof of the following two inequalities:
\begin{multline*}
\frac{(q+V(y)+A(\eta))^{\frac{1}{\kappa}}}{(q+V(x)+A(\xi))^{\frac{1}{\kappa}}}\leq \\ \leq
 C\left(1+(q+V(x)+A(\xi))^{\frac{1}{\gamma}}|x-y|^2+(q+V(x)+A(\xi))^{\frac{1}{\kappa}}|\xi-\eta|^2\right)^{N},
   \end{multline*}
and
 \begin{multline*}
 \frac{(q+V(y)+A(\eta))^{\frac{1}{\gamma}}}{(q+V(x)+A(\xi))^{\frac{1}{\gamma}}} \leq\\ \leq
 C\left(1+(q+V(x)+A(\xi))^{\frac{1}{\gamma}}|x-y|^2+(q+V(x)+A(\xi))^{\frac{1}{\kappa}}|\xi-\eta|^2\right)^{N}.
   \end{multline*}
We note that it is enough to prove the first inequality, which, in turn, can be reduced to the following one 
\[\frac{(q+V(y)+A(\eta))^{\frac{1}{2\kappa}}}{(q+V(x)+A(\xi))^{\frac{1}{2\kappa}}}\leq\]
 \beq\label{tempk4}\leq C\left(2+(q+V(x)+A(\xi))^{\frac{1}{2\gamma}}|x-y|+(2+(q+V(x)+A(\xi))^{\frac{1}{2\kappa}}|\xi-\eta|\right)^{N}.\eq
Now \eqref{tempk4} can be obtained from the inequalities below
\beq\label{tempk4d}\frac{(\qt+V(y))^{\frac{1}{2\kappa}}}{(q+V(x)+A(\xi))^{\frac{1}{2\kappa}}}\leq \frac{(\qt+V(y))^{\frac{1}{2\kappa}}}{(\qt+V(x))^{\frac{1}{2\kappa}}}\leq
 C\left(2+(q+V(x))^{\frac{1}{2\gamma}}|x-y|\right)^{N},\eq
 \beq\label{tempk4a}\frac{(\qt+A(\eta))^{\frac{1}{2\kappa}}}{(q+V(x)+A(\xi))^{\frac{1}{2\kappa}}}\leq \frac{(\qt+A(\eta))^{\frac{1}{2\kappa}}}{(\qt+A(\xi))^{\frac{1}{2\kappa}}}\leq
 C\left(2+(\frac{q}{2}+A(\xi))^{\frac{1}{2\kappa}})|\xi-\eta|\right)^{N}.\eq
In order to verify \eqref{tempk4d} we observe that, by choosing an integer $N$ such that $\frac{N}{2\gamma}\geq 1$, and applying the assumption (b) of Definition \ref{class1a} we obtain for a suitable $C>0$:
\[(\qt+V(y))^{\frac{1}{2\kappa}}\leq C\left(|x-y|+(\qt+V(x))^{\frac{1}{2\kappa}}\right)\leq \]\[ \leq C\left(1+(\qt+V(x))^{\frac{1}{2\gamma}}|x-y|\right)^{N}(1+V(x))^{\frac{1}{2\kappa}}\]
The proof of inequality \eqref{tempk4a} follows similarly, and the poof of Theorem \ref{contgk} is complete.
\end{proof}

We have shown that $g=g^{(A,V)}$ is a H\"ormander metric,thus the $S(M,g)$ classes are defined for any $g$-weight $M$ and a  Weyl-H\"ormader calculus is available to our disposal.

We now make some essential observations regarding our symbol $V(x)+A(\xi)$.
\begin{thm} Let $g=g^{(A,V)}$ be the metric defined  by \eqref{anhmet012x}. Then
\begin{itemize}
\item[(i)] $q+V(x)+A(\xi)$ is a $g$-weight.
\item[(ii)] There exists a regular weight $m$  equivalent to 
$M(X)=q+V(x)+A(\xi)$, i.e., we have $\left(\frac{m(Y)}{M(Y)}\right)^{\pm 1}\leq C$, for all $Y$, and $m\in S(m,g)$. Consequently,  $S(m,g)=S(M,g)$.
\end{itemize}
\end{thm}
\begin{proof} (i) A look at the proof of the continuity and temperateness of the metric $g$ shows that $q+V(x)+A(\xi)$ is a $g$-weight. We also give a different and more general argument showing this fact.  

For the metric  $g=g^{(A,V)}$ the uncertainty parameter is given by 
\[\lambda_{g}(X)=(q+V(x)+A(\xi))^{\frac{\kappa+\gamma}{2\kappa\gamma}}.\]
In general the uncertainty parameter of a H\"ormander metric $G$ is a $G$-weight (cf. \cite{Hor85} ). So, by taking the power
$\frac{2\kappa\gamma}{\kappa+\gamma}$ of $\lambda_{g}$, we obtain that $q+V(x)+A(\xi)$ is a $g$-weight.
(ii)  For any weight $M$ one can construct an equivalent regular weight
  $\widetilde{M}$ as a consequence of the existence of suitable $g$-partitions of unity (cf. \cite{Ler}, \cite{BL}, \cite{BC}). Since $q+V(x)+A(\xi)$
is a $g$-weight, in particular, there exists a regular weight $m$  equivalent to  $M=q+V(x)+A(\xi)$. It is clear that $S(m,g)=S(M,g)$. 
\end{proof}
For some purposes it is enough to work with the equivalent weight $\widetilde{M}$. 
 In our case with the perspective of establishing Schatten-von Neumann properties we  will assume some further regularity conditions that will be introduced in Section \ref{sec:Schatten}.In particular, it will be assumed  that  $A(\xi)+V(x)$ is a symbol in the class $S(q+V(x)+A(\xi),g^{(A,V)} )$, a fact that holds for the main examples we are going to present.\\ 

\begin{ex}
As an example we observe the relativistic Schr\"odinger operator  $\sqrt{I-\Delta}+V(x)$ with $V$ satisfying  the Definition \ref{class1a} for some $\kappa>0$ and $V\in S(q+V(x)+A(\xi),g^{(A,V)} )$ with $A(D)=\sqrt{I-\Delta}$. In this case we have $\gamma=\half$ and 
\[\jp+V(x)\in S(\lambda_g^{\frac {\kappa}{\kappa +\half}},g^{(A,V)}),\]
i.e. the relativistic Schr\"odinger operator is of order $\frac {\kappa}{\kappa +\half}$ with respect to $g^{(A,V)}$. 
\end{ex}

\section{An intrinsic definition of the classes $S(\lambda_g^{m},g)$}
\label{SEC:anharmonic-notes}
Let $g=g^{(A,V)}$ be the H\"ormander metric  defined by \eqref{anhmet012x}. Since we have already a guarantee of the existence of a corresponding pseudo-differential calculus, we can now  define the class $S(\lambda_g^{m},g)$, for $m\in\ar$, in an equivalent and intrinsic  way without referring  explicitly to the metric $g$. 
\begin{defn}\label{DEF:sigmas} 
Let $m\in\ar$, and let $A(\xi), V(x)$ be a $\gamma$-function and a $\kappa$-function, respectively. For $a\in C^{\infty}(\ardn)$, we will say that $a\in\Sigma_{A,V}^m$, if for some $q>0$ large enough, and for all multi-indices $\alpha,\beta$ there exists a constant $C_{\alpha\beta}$ such that
 \beq\label{sigmacl}|\partial_{x}^{\beta}\partial_{\xi}^{\alpha}a(x,\xi)|\leq C_{\alpha\beta}(q+V(x)+A(\xi))^{m-\frac{|\beta|}{2\kappa}-\frac{|\alpha|}{2\gamma}}\,,\eq
 for all $x,\xi\in\mathbb R^n$.
\end{defn}
We point out that the definition above corresponds indeed to the  one for the metric $g^{(A,V)}$:
\begin{prop}\label{sht57} Let $m\in\ar$ and let $g=g^{(A,V)}$ be the metric defined by \eqref{anhmet012x}. Then 
\[\Sigma_{A,V}^m=S(\lambda_g^{m(\frac{2\kappa\gamma}{\kappa+\gamma})},g).\]
\end{prop}

\begin{proof} Observe that for  $g=g^{(A,V)}$  we have  
\[q+V(x)+A(\xi)\sim \lambda_{g}^{\frac{2\kappa\gamma}{\kappa+\gamma}}.\] Therefore, by \eqref{sigmacl}, we have
\[S(\lambda_g^{m(\frac{2\kappa\gamma}{\kappa+\gamma})},g)\supset \Sigma_{A,V}^m.\] 
	On the other hand, if $a \in C^{\infty}(\arn \times \arn) \in S(\lambda_g^{m(\frac{2\kappa\gamma}{\kappa+\gamma})},g)$, then by Definition \ref{inwhk} and by taking canonical directional derivatives for  every pair of  multi-indices $\alpha,\beta$ there exists a constant $C_{\alpha\beta}>0$ such that
	
\[|\partial_{x}^{\beta}\partial_{\xi}^{\alpha}a(x,\xi)|\leq C_{\alpha\beta}(q+V(x)+A(\xi))^{m-\frac{|\beta|}{2\kappa}-\frac{|\alpha|}{2\gamma}}\,,\]
for all $x,\xi\in\mathbb R^n$.	\\
Therefore $S(\lambda_g^{m(\frac{2\kappa\gamma}{\kappa+\gamma})},g) \subset \Sigma_{A,V}^m$, and this completes the proof.
\end{proof}
We  now consider the composition formula in the setting of the classes $\Sigma_{A,V}^m$ as a consequence of the corresponding one in the $S(M,g)$ calculus. 
\begin{thm}\label{thm.comp} Let $m_1, m_2\in\ar$. If $a\in \Sigma_{A,V}^{m_1}, b\in\Sigma_{A,V}^{m_2}$. There exists $c\in \Sigma_{A,V}^{m_1+m_2}$ such that $a(x,D)\circ b(x,D)=c(x,D)$ and
\[c(x,\xi)\sim \sum\limits_{\alpha}(2\pi i)^{-|\alpha|}\partial_{\xi}^{\alpha}a(x,\xi)\partial_{x}^{\alpha}b(x,\xi),\]
i.e., for all $N\in\ene$
\[c(x,\xi)- \sum\limits_{|\alpha|<N}(2\pi i)^{-|\alpha|}\partial_{\xi}^{\alpha}a(x,\xi)\partial_{x}^{\alpha}b(x,\xi)\in \Sigma_{k,\ell}^{m_1+m_2-N\left(\frac{\kappa+\gamma}{2\kappa\gamma}\right)} .\]
\end{thm}

The  $L^2$ boundedness of operators in the class $\Sigma_{A,V}^{0}$ is celarly obtained from the corresponding one for the $S(1,g)$ class.
\begin{thm} Let $a\in\Sigma_{A,V}^{0}$. Then, $a(x,D)$ extends to a bounded operator $a(x,D):L^2(\arn)\rightarrow L^2(\arn)$. 
\end{thm}

\section{Schatten-von Neumann classes and spectral properties}\label{sec:Schatten}

We are now ready to study some properties of negative powers of operators of the form $q+A(D)+V(X)$ in relation with their behaviour in the Schatten-von Neumann classes of operators. The main aim of this section is to provide a simple proof for the main term of the spectral asymptotics of  these operators.  In the context of the Weyl-H\"ormander calculus it is useful to recall the following result by Toft  \cite{Tof06}. See also  
 \cite{BT} for further developments. 
\begin{thm}\label{toft} Let  $g$ be a split H\"ormander metric, and let $M$ be a $g$-weight. For $1\leq r <\infty$, if $h_g^{\frac{k}{2}}M\in L^r(\ardn)$ for some $k\geq 0$, then, for any $a\in S(M,g)$, we have 
\[a_t(x,D)\in S_r(L^2(\arn)), \,\,\mbox{ for all }t\in\ar.\] 
\end{thm} 
In order to obtain the desired spectral properties for the negative powers of our operator $q+A(D)+V(x)$, we are going to assume two additional conditions, that hold for relevant cases, and in particular for our main examples. Precisely, the following definition introduces these extra assumptions: 

\begin{defn}\label{tfcond1}
We define the following conditions on the pair of functions $(A(x),V(\xi))$:\\

 \begin{itemize} 
 \item[(C1)]  $A,V$ are $C^{\infty}$ functions on $\Rn$ and $A(\xi)+V(x)\in \Sigma_{A,V}^{1}.$ 
 \item[(C2)] There exist $\delta>0$ and $C>0$ such that $q+V(x)+A(\xi)\geq C(1+|x|+|\xi|)^{\delta}$. 

 \item[(C3)] There exists $\mu>0$  such that 
 \beq\label{intcf8}\int\limits_{\arn}\int\limits_{\arn}(q+V(x)+A(\xi))^{-\mu}<\infty.\eq 
 \end{itemize}
\end{defn}
\begin{rem} Let us highlight, that condition (C2) implies condition (C3); indeed, one can choose $\mu >\frac{2n}{\delta}$ in (C3), where $\delta$ is the one appearing in (C2). In particular, given condition (C1), condition (C2) implies, one the one hand, that the symbol $q+V(x)+A(\xi)$ is $g$-elliptic with respect to the metric $g=g^{(A,V)}$ in the sense of  \cite{BN}, and, on the other hand, the integrability condition \eqref{intcf8}. However, condition \eqref{intcf8} can also be achieved for $\mu\leq \frac{2n}{\delta}$. This is why we  chose to refer to condition (C3) explicitly, and not  to include it in condition (C2). In particular, as we will see for the fractional relativistic Schr\"odinger operators $(I-\Delta)^\gamma+\langle x \rangle^{2k}$,  one can choose in condition (C3) some $\mu>\mu_0$, where $\mu_0=\frac{n(\kappa+\gamma)}{2\kappa\gamma}.$ For this type of operators the natural choice for  $\delta$ is $\delta=\min\{2\kappa, 2\gamma\}$ in order to guarantee (C2). Now, comparing the lower bounds, $\frac{2n}{\delta}$ and $\frac{n(\kappa+\gamma)}{2\kappa\gamma}$, for $\mu$, we see that, for $\kappa\neq \gamma$, we have that   
\[\frac{2n}{\delta}=\frac{2n}{\min\{2\kappa, 2\gamma\}}>\frac{n(\kappa+\gamma)}{2\kappa\gamma}.\]\\
In the case where $\kappa=\gamma$, the two lower bounds are identical. 

Therefore the exponent $\mu_0=\frac{n(\kappa+\gamma)}{2\kappa\gamma},$  gives a better condition for $\mu$ in (C3) and moreover one can show that it is sharp; see Theorem \ref{sch1mfff} and Remark \ref{rem.sharp}.
\end{rem}

Regarding the condition (C1), we give a mild sufficient condition guaranteeing such membership.

\begin{lem}\label{lemsplit9} Let $m\in\ar$, and let $A, V$ be a $\gamma$-function and a $\kappa$-function, respectively. Let $B=B(x), W=W(\xi)$ be $C^{\infty}$ complex-valued functions defined on $\Rn$ such that:\\

(i) for all multi-indices $\beta$ there exists a constant $C_{\beta}$ such that
 \beq\label{sigy7}|\partial_{x}^{\beta}B(x)|\leq C_{\beta}(q+V(x)+A(\xi))^{m-\frac{|\beta|}{2\kappa}}\, ,\mbox{ for all } x\in\mathbb R^n\,;\eq

(ii)  for all multi-indices $\alpha$ there exists a constant $C_{\alpha}$ such that
 \beq\label{sigy9}|\partial_{\xi}^{\alpha}W(\xi)|\leq C_{\alpha}(q+V(x)+A(\xi))^{m-\frac{|\alpha|}{2\gamma}}\, ,\mbox{ for all } \xi\in\mathbb R^n.\eq

Then, $W(\xi)+B(x)\in \Sigma_{A,V}^{m}$.
\end{lem} 
\begin{proof} The assumptions (i) and (ii) ensure that $B$ and $W$ belong to $\Sigma_{A,V}^{m}$ according to Definition \ref{DEF:sigmas}. Hence, by the general theory, we also have $W(\xi)+B(x)\in \Sigma_{A,V}^{m}$.
\end{proof}
\begin{ex}
We note that the pairs $(A(x),V(\xi))$ listed in {\em Example} \ref{exsy3},   also  satisfy the conditions given by Definition \ref{tfcond1}. Indeed, the condition (C1) follows from Lemma \ref{lemsplit9}, while the condition (C2) holds by choosing $\delta=\min\{2\kappa, 2\gamma\}$. Condition (C3) then follows from condition (C2) as explained above.
For instance, let us consider the fractional relativistic Schr\"odinger operator $(I-\Delta)^{\gamma}+\jpx^{2\kappa}$,  where $\gamma >0$. We take 
$A(\xi)=\jp^{2\gamma}, V(x)=\jpx^{2\kappa}$ and  observe that  $V(x), A(\xi)$ satisfy (i) and (ii), respectively, in Lemma \ref{lemsplit9} with $m=1$. Therefore  $\jp^{2\gamma}+\jpx^{2\kappa}\in \Sigma_{\jp^{2\gamma},\jpx^{2\kappa}}^{1}$. 
\end{ex}
\medskip

\noindent We can now formulate some consequences for Schatten-von Neumann classes:
\begin{thm}\label{sch1mf} Let $g=g^{(A,V)}$ be the metric defined  in \eqref{anhmet012x},  $1\leq r<\infty$, and assume that (C3) holds for some $\mu>0$ as in \eqref{intcf8}.  Then, for $a\in S((q+V(x)+A(\xi))^{-\frac{\mu}{r}},g)$, we have
 \[a_t(x,D)\in S_r(L^2(\arn)), \,\,\mbox{ for all }t\in\ar.\]
\end{thm}
\begin{proof} We observe that condition \eqref{intcf8} on $\mu$
implies that 
 \[(q+V(x)+A(\xi))^{-\frac{\mu}{r}}\in L^r(\ardn).\]
 Hence, an application of Theorem \ref{toft}  with $M=(q+V(x)+A(\xi))^{-\frac{\mu}{r}},$ and for $k=0$, yields that for $a\in S((q+V(x)+A(\xi))^{-\frac{\mu}{r}},g)$, we have
\[a_t(x,D)\in S_r(L^2(\arn)), \,\,\mbox{ for all }t\in\ar\,,\]
proving the theorem.
\end{proof}
We now recall some basic properties regarding the  Hilbert-Schmidt and trace class operators, see e.g. \cite{NR}.
\begin{rem}\label{NR.rem}
		If $a^w$ extends to a Hilbert-Schmidt operator on $L^2(\mathbb{R}^n)$, then one has 	\[
		\|a^w\|_{S_2}=(2\pi)^{-\frac{n}{2}}\|a\|_{L^2(\mathbb{R}^{2n})}\,,
		\] 
		while if $a^w$ extends to a trace-class operator on $L^2(\mathbb{R}^n)$, then one has 
		\[
		\textnormal{Tr}(a^w)=(2\pi)^{-n}\int_{\mathbb{R}^{2n}}a(x,\xi)dx\,d\xi\,.
		\]
	\end{rem}
We can now obtain some consequences for the negative powers of our Hamiltonians.
 \begin{cor}
		\label{cor.b.1} Let $g=g^{(A,V)}$ be the metric defined by  \eqref{anhmet012x}.
		\begin{enumerate}[label=(\alph*)]
     \item \label{itm:cor.sch.a} Let $a\in S((q+V(x)+A(\xi))^{-\nu},g)$. Then, the following hold true:
			\begin{enumerate}[label=(\roman*)]
				\item For all $\nu \geq \frac{\mu}{2}$, where $\mu$ is the one appearing in \eqref{intcf8}, and for all $t \in \ar$,  the operator $a_t(x,D)$ extends to a Hilbert-Schmidt operator on $L^2(\mathbb{R}^n)$, and in particular we have
				\[
				\|a_t(x,D)\|_{S_2}=(2\pi)^{-\frac{n}{2}}\|a\|_{L^2(\mathbb{R}^{2n})}\,.
				\]
			
				\item  For all $\nu \geq \mu$,where $\mu$ is the one appearing in \eqref{intcf8}, and  for all $t \in \mathbb{R}$, the operator $a_t(x,D)$, extends to a trace-class operator on $L^2(\mathbb{R}^n)$, and in particular we have
				\[
				\textnormal{Tr}(a_t(x,D))=(2\pi)^{-n}\int_{\mathbb{R}^n}\int_{\mathbb{R}^n}a(x,\xi)\,dxd\xi\,.
				\]
					\end{enumerate}
			
\item \label{itm:cor.sch.11} Let us additionally assume that the pair $(A(x), V(\xi))$ satisfies the conditions (C1), (C2), and (C3) for some $\mu>0$. Then, for $1 \leq r<\infty$ and  $\nu\geq \frac{\mu}{r}$, we have 
 \[T^{-\nu}\equiv (q+A(D)+V(x))^{-\nu}\in S_r(L^2(\arn))\,.\]
 	
			\item \label{itm:cor.sch.b} We keep the pair $(A(x), V(\xi))$ satisfying the conditions (C1), (C2) and (C3) for  some $\mu>0$. Then, the following hold true:
			\begin{enumerate}[label=(\roman*)]
				\item The operator $T^{-\frac{\mu}{2}}$ extends to a Hilbert-Schmidt operator on $L^2(\mathbb{R}^n)$, and there exists a constant $C>0$ such that 		\[
				\|T^{-\frac{\mu}{2}}\|_{S_2}\leq C \|(q+V(x)+A(\xi))^{-\frac{\mu}{2}}\|_{L^2(\mathbb{R}^{2n})}\,.
				\] 
				\item The operator $T^{-\mu}$ extends to a trace-class operator on $L^2(\mathbb{R}^n)$, and there exists a constant $C>0$ such that
				\[
				|\textnormal{Tr}(T^{-\mu})|\leq C \int_{\mathbb{R}^n}\int_{\mathbb{R}^n}(q+V(x)+A(\xi))^{-\mu}\,dxd\xi\,.
				\]
			\end{enumerate}
		\end{enumerate}
	\end{cor}
\begin{proof}[Proof of Corollary \ref{cor.b.1}]  The proof of part \ref{itm:cor.sch.a} follows by Theorem \ref{sch1mf} for $r=1,2$ and by Remark \ref{NR.rem}, since, as mentioned above, condition (C2) implies condition (C3). We now prove part \ref{itm:cor.sch.11}. We note that, on the one hand, we have $A(D)=A^w(x,D)$, while on the other hand, the multiplication operator $V(x)$ has the potential $V(x)$ itself as is Kohn-Nirenberg symbol, with its Weyl symbol being equal to $V(x)$+ {\em lower order}. The last is due to the asymptotic formula for comparing  the corresponding symbols with respect to different $t$-quantizations (cf. formula (2.3.29) of Theorem 2.3.18 in  \cite{Ler}). Moreover, the $g$-ellipticity of the symbol $q+V(x)+A(\xi)$ ensures that, the negative powers $(q+A(D)+V(x))^{-\nu}$ for all $\nu>0$,  are well defined, and their symbols belong to the class $S((q+V(x)+A(\xi))^{-\nu},g)$. Now an  application of Theorem  \ref{sch1mf} to the symbol of  $T^{-\nu} $ concludes the proof of part \ref{itm:cor.sch.11}.
The proof of  part \ref{itm:cor.sch.b} follows from part  \ref{itm:cor.sch.11}, and Remark \ref{NR.rem}, since, arguing as in the proof of part \ref{itm:cor.sch.11}, we see that the symbol $\sigma_1$ (resp. $\sigma_2$) of $T^{-\frac{\mu}{2}}$ (resp. $T^{-\mu}$) is in the class $((q+V(x)+A(\xi))^{-\frac{\mu}{2}},g)$ (resp. $S((q+V(x)+A(\xi))^{-\mu},g))$. The latter means that we, there exists some constant $C$ such that 
	\[
		|\sigma_1(x,\xi)|\leq C(q+V(x)+A(\xi))^{-\mu}\, ,\quad (\text{resp.}\,\quad \sigma_2 \leq C(q+V(x)+A(\xi))^{-\frac{\mu}{2}})\,.
		\]

			The proof is now complete.
	\end{proof}
	We now derive some consequences for the singular values of the operators we considered. In the sequel we denote by $\lambda_j(K)$ the eigenvalues, in decreasing order, of the compact operator $K$ on the space $H$ as above, and by $s_j(K)$ the singular values of it. 
\begin{cor}\label{cor.eigen.asympt.anharm.}
	Let $g=g^{(A,V)}$ be the metric defined by  \eqref{anhmet012x} and let $1 \leq r < \infty$.
	\begin{enumerate}[label=(\roman*)]
		\item \label{itm:eig.opw}  If $a \in S((q+V(x)+A(\xi))^{-\nu},g)$, assume (C3) holds for some $\mu>0$ as in \eqref{intcf8} and let $\nu\geq\frac{\mu}{r}$. Then 
		\[
		s_j(a_t(x,D))=o(j^{-\frac{1}{r}})\,,\quad \text{as} \quad j \rightarrow \infty\,,
		\]
		for every $t \in \ar$.
		Consequently, also
		 	\[
		 \lambda_j(a_t(x,D))=o(j^{-\frac{1}{r}})\,,\quad \text{as} \quad j \rightarrow \infty\,,
		 \]
		 for every $t \in \ar$.
		\item \label{itm:eig.anh.pr} If we assume that the pair $(A,V)$ satisfy conditions (C1) and (C2), then as a particular case of \ref{itm:eig.opw}, we get that for $T^{-\nu}=(q+A(D)+V(x))^{-\nu}$, we have that
		\begin{equation}\label{eig.asym}
			\lambda_j(T^{-\nu}) =o(j^{-\frac{1}{r}})\,,\quad \text{as}\quad j \rightarrow \infty\,,
		\end{equation} 
		for all $\nu \geq \frac{\mu}{r}$, where $1 \leq r < \infty$.

	\end{enumerate}
\end{cor}
\begin{proof}
	As Corollary \ref{cor.b.1} implies we have $a_t(x,D)\in S_r(\arn)$ for $\nu$ in the range as in the statement. Hence, as it is well known, one can get for the membership of the above class that 
	\[
	s_j(a_t(x,D))=o(j^{-\frac{1}{r}})\,,\quad \text{as}\quad j \rightarrow \infty\,.	\]
	Moreover from the Weyl inequality one has
		\[
	\lambda_j(a_t(x,D))=o(j^{-\frac{1}{r}})\,,\quad \text{as}\quad j \rightarrow \infty\,.	\] This proves \ref{itm:eig.opw}, while similar arguments can be used to prove \ref{itm:eig.anh.pr} having into account that $T^{-\nu}$ is positive definite.
\end{proof}
We now derive an immediate consequence on the rate of growth of the eigenvalues. 
 First we note that by Corollary \ref{cor.eigen.asympt.anharm.} (ii), for $\nu=1$ and $r\geq\mu$ we get 
 \beq\lambda_j((q+A(D)+V(x))^{-1})=\os(j^{-\frac{1}{r}}),\, \mbox{ as }j\rightarrow\infty .\eq
From this  we obtain the following estimate for the rate of growth of the eigenvalues of $A(D)+V(x)$:\\

For every $L\in\ene$ there exists $L_0\in\ene$ such that
\beq
Lj^{\frac{1}{r}}\leq\lambda_j(A(D)+V(x)),\,\, \mbox{ for }j\geq L_0. 
\eq
Summarising we have obtained the following: 
Let  $r\geq\mu$. Then,  
for every $L\in\ene$ there exists $L_0\in\ene$ such that
\beq
Lj^{\frac{1}{r}}\leq\lambda_j(A(D)+V(x)),\,\, \mbox{ for }j\geq L_0. 
\eq
Thus, the eigenvalues $\lambda_j(A(D)+V(x))$ have a growth of  order at least
\beq \label{EQ:growthb}
j^{\frac{1}{r}}, \mbox{ as } j\rightarrow\infty.
\eq
We now consider a special class of pairs $(A,V)$   
 satisfying the following condition :

Let $A, V$ satisfy Definition \ref{class1a} respectively for $\gamma, \kappa$.

(C*) There exist constants $C_1, C_2>0$ such that:
\[A(\xi)\geq C_1|\xi|^{2\gamma}\,\,\, {\rm for }\,\,\,|\xi|\geq C_2\,; \,\, V(x)\geq C_1|x|^{2\kappa}\,\,\ {\rm for }\,\, |x|\geq C_2\, .\]

We observe that condition (C*) implies (C2). Moreover, will see that for this type of pairs we can obtain a sharp $\mu$ satisfying (C3). We also note that all pairs $(A,V)$ taken from the list we provided in the Example \ref{exsy3} satisfy (C*). 

\begin{thm} \label{sch1mfff} Let  $(A,V)$ be a pair satisfying condition (C*). Then we have 
\begin{equation}
    \mathcal{I}_{\mu}:=\int\limits_{\arn}\int\limits_{\arn}(q+A(\xi)+V(x))^{-\mu}dxd\xi<\infty,
\end{equation}
provided $\mu>\frac{n(\kappa+\gamma)}{2\kappa\gamma}$.\\

Consequently, if additionally (C1) holds, and $T^{-\nu}=(q+A(D)+V(x))^{-\nu},$ we have
\[T^{-\nu}\in  S_r(L^2(\arn)),\]
provided $\nu>\frac{n(\kappa+\gamma)}{2r\kappa\gamma}$  for $1\leq r<\infty$.
\end{thm}
\begin{proof} Since the functions $A, V$ are continuous,  they are measurable, and by (C*) there exists a constant $C>0$ such that:
\begin{equation}\label{inh6s}
    \mathcal{I}_{\mu}=\int\limits_{\arn}\int\limits_{\arn}(q+A(\xi)+V(x))^{-\mu}dxd\xi\geq C\int\limits_{\arn}\int\limits_{\arn}(1+|x|^{2\kappa}+|\xi|^{2\gamma})^{-\mu}dxd\xi
\end{equation}

Now the integral 
 \[\int\limits_{\arn}\int\limits_{\arn}(1+|x|^{2\kappa}+|\xi|^{2\gamma})^{-\mu}dxd\xi\]
 can be estimated on the subset $B=\{(x,\xi)\in\ardn: \xi_i >0\, , i=1,\dots, n\,\}$. Without loss of generality, we can assume that 
 $\kappa\geq \gamma,$ . The change of variable in $B$ given by $(x,\xi)\rightarrow (x_1,\dots,x_n,\xi_1^{\frac{\kappa}{\gamma}},\dots,\xi_n^{\frac{\kappa}{\gamma}})$ lead us to 
 
\begin{align*} \int\limits_{\arn}\int\limits_{\arn}(1+|x|^{2\kappa}+|\xi|^{2\gamma})^{-\mu}dxd\xi=&C\int\limits_{B}(1+|x|^{2\kappa}+|\xi|^{2\kappa})^{-\mu}\xi_1^{\frac{\kappa}{\gamma}-1}\cdots\xi_n^{\frac{\kappa}{\gamma}-1}dxd\xi\\
\leq& \,C\int\limits_{B}(1+|x|^{2}+|\xi|^{2})^{-\mu\kappa}|\xi|^{(\frac{\kappa}{\gamma}-1)n}dxd\xi\\
\leq& \,C\int\limits_{B}(1+|x|^{2}+|\xi|^{2})^{-\mu\kappa}|(x,\xi)|^{(\frac{\kappa}{\gamma}-1)n}dxd\xi\\
\leq& \,C\int\limits_{B}\langle X\rangle^{(\frac{\kappa}{\gamma}-1)n-2\mu\kappa}dxd\xi\\
< & \infty\,,
\end{align*}
provided  $(\frac{\kappa}{\gamma}-1)n-2\mu\kappa<-2n.$ Now, the condition $(\frac{\kappa}{\gamma}-1)n-2\mu\kappa<-2n$ for the convergence of the integral is reduced to 
$\mu >\frac{n(\kappa+\gamma)}{2\kappa\gamma }$.
 
For the second part, since (C*) implies (C2), and (C3) holds with $\mu >\frac{n(\kappa+\gamma)}{2\kappa\gamma }$ from the above estimate, an application of Corollary \ref{cor.b.1}  \ref{itm:cor.sch.11} concludes the proof of the theorem, since $\nu >\frac{n(\kappa+\gamma)}{2r\kappa\gamma }$ implies  $r\nu >\frac{n(\kappa+\gamma)}{2\kappa\gamma }$. 
\end{proof}
\begin{rem}\label{rem.sharp}
It is clear that the condition $\mu>\frac{n(\kappa+\gamma)}{2r\kappa\gamma }$ is sharp, as it can be tested on the case of the symbol of the harmonic oscillator and by using the classical integrability criteria for negative powers of $\langle X\rangle$ of dimension $2n$. 
\end{rem}
\begin{ex}\label{ex.rel.sch}
Theorem \ref{sch1mfff} applied to the particular case of the fractional relativistic Schr\"odinger operator $(I-\Delta)^\gamma+\langle x \rangle^{2k}$, where $\gamma>0$, yields $((I-\Delta)^\gamma+\langle x \rangle^{2k})^{-\nu}\in S_r(L^2(\mathbb{R}^n))$, for $\nu >\frac{n(\kappa+\gamma)}{2r\kappa\gamma }$.
\end{ex}
\begin{rem}
Still regarding the example of the fractional relativistic Schr\"odinger operator as above, we note that the condition in Example \ref{ex.rel.sch}, can also be obtained if one uses a different line of arguments: if $N(\lambda)$ is the eigenvalue counting function of $(I-\Delta)^\gamma+\langle x \rangle^{2k}$, then, by Theorem 3.2 of \cite{BR}, we have $N(\lambda) \lesssim \int \int_{\langle \xi \rangle^{2\gamma}+\langle x \rangle^{2\kappa}<\lambda}dx\,d\xi$ for large values of $\lambda$. Indeed, by the change of variable $\xi=\lambda^{\frac{1}{2\gamma}}\xi'$ and $x=\lambda^{\frac{1}{2\kappa}}x'$ we can estimate
\begin{eqnarray*}
N(\lambda) & \lesssim & \int \int_{\langle \xi \rangle^{2\gamma}+\langle x \rangle^{2\kappa}<\lambda}dx\,d\xi \leq \int \int_{| \xi |^{2\gamma}+| x |^{2\kappa}<\lambda}dx\,d\xi\\
& = & \lambda^{n \left(\frac{1}{2\gamma}+\frac{1}{2\kappa} \right)}\int \int \int_{| \xi' |^{2\gamma}+| x' |^{2\kappa}<1}dx'\,d\xi' \lesssim \lambda^{n \left(\frac{1}{2\gamma}+\frac{1}{2\kappa} \right)}\,.
\end{eqnarray*}
From the last estimate one can deduce that $((I-\Delta)^\gamma+\langle x \rangle^{2\kappa})^{-\nu}\in S_r(L^2(\mathbb{R}^n))$ for $\nu>\frac{n(\kappa+\gamma)}{2\kappa \gamma r}$, and our claim follows. 
\end{rem}

\end{document}